\documentclass[12pt]{amsart}
\usepackage{amssymb}
\theoremstyle{plain}
 \newtheorem{thm}{Theorem}[section]
 \newtheorem{cor}[thm]{Corollary}
 \newtheorem{lem}[thm]{Lemma}
 \newtheorem{prop}[thm]{Proposition}

\theoremstyle{definition}
 \newtheorem{ex}{Example}[section]
\theoremstyle{remark}
 \newtheorem{rem}{Remark}[section]
\begin{document}
\title[On normalizers of $C^{*}$-subalgebras in the 
Cuntz algebra $\mathcal{O}_{n}$]
{On normalizers of $C^{*}$-subalgebras in the 
Cuntz algebra $\mathcal{O}_{n}$}
\author[Tomohiro Hayashi]{{Tomohiro Hayashi} }
\address[Tomohiro Hayashi]
{Nagoya Institute of Technology, 
Gokiso-cho, Showa-ku, Nagoya, Aichi, 466-8555, Japan}
\email[Tomohiro Hayashi]{hayashi.tomohiro@nitech.ac.jp}
\baselineskip=17pt

\maketitle

\begin{abstract}
In this paper we investigate the normalizer 
$\mathcal{N}_{\mathcal{O}_{n}}(A)$ of a 
$C^{*}$-subalgebra $A\subset \mathcal{F}_{n}$ where 
$\mathcal{F}_{n}$ is the canonical UHF-subalgebra of type $n^{\infty}$ 
in the Cuntz algebra $\mathcal{O}_{n}$. Under the assumption 
that the relative commutant 
$A'\cap \mathcal{F}_{n}$ is finite-dimensional, we show 
several facts for normalizers of $A$. In particular it is 
shown that 
the automorphism group 
$\{{\rm Ad}u|_{A}\ \ |\ u\in \mathcal{N}_{\mathcal{F}_{n}}(A)\}$ 
has a finite index in 
$\{{\rm Ad}U|_{A}\ \ |\ U\in \mathcal{N}_{\mathcal{O}_{n}}(A)\}$.

\end{abstract}

\section{Introduction}
The purpose of this paper is to investigate the normalizer 
of $C^{*}$-subalgebras in the Cuntz algebra 
$\mathcal{O}_{n}$~\cite{C}. 
Let $\mathcal{F}_{n}$ be the canonical UHFsubalgebra of 
$\mathcal{O}_{n}$. In the paper~\cite{CHS}, 
it is shown that the normalizer group 
$\mathcal{N}_{\mathcal{O}_{n}}(\mathcal{F}_{n})$ is a 
subset of $\mathcal{F}_{n}$. 
(In~\cite{CHS}, more general results are shown.) 
More generally, if $A$ is an irreducible 
$C^{*}$-subalgebra of $\mathcal{O}_{n}$, then 
the normilizers 
$\mathcal{N}_{\mathcal{O}_{n}}(A)$ is a subset of $\mathcal{F}_{n}$. 
In this paper we investigate the normalizer 
$\mathcal{N}_{\mathcal{O}_{n}}(A)$ where $A$ is a 
$C^{*}$-subalgebra of $\mathcal{F}_{n}$ with a 
finite-dimensional relative commutant in $\mathcal{F}_{n}$. 
In this setting the normalizer $\mathcal{N}_{\mathcal{O}_{n}}(A)$ is 
not a subset of $\mathcal{F}_{n}$ in general. 
However we can show that the inner automorphism group induced by 
the elements in $\mathcal{N}_{\mathcal{F}_{n}}(A)$ has a 
finite index in the inner automorphism group induced by 
the elements in $\mathcal{N}_{\mathcal{O}_{n}}(A)$.  
In order to show this fact, we show that 
the relative commutant 
$A'\cap \mathcal{O}_{n}$ is also finite-dimensional. 
As a corollary of our investigation, it is shown 
that the irreducibility $A'\cap \mathcal{F}_{n}={\Bbb C}$ 
implies that 
$A'\cap \mathcal{O}_{n}={\Bbb C}$. Hence in this case 
the normalizer group 
$\mathcal{N}_{\mathcal{O}_{n}}(A)$ is a subset of $\mathcal{F}_{n}$.

We would like to explain the motivation of this paper. 
There is a one-to-one correspondence 
between all unitaries $\mathcal{U}(\mathcal{O}_{n})$ and 
all endomorphisms ${\rm End}(\mathcal{O}_{n})$ 
such that 
$$
\mathcal{U}(\mathcal{O}_{n})\ni u\mapsto 
\lambda_{u}\in {\rm End}(\mathcal{O}_{n})
$$
where $\lambda_{u}$ is defined by 
$\lambda_{u}(S_{i})=uS_{i}$. The endomorphism $\lambda_{u}$ 
is called {\it localized} if the corresponding unitary is a 
matrix in the UHF-algebra $\mathcal{F}_{n}$
~\cite{CP,CF}. 
In the 
paper~\cite{S} Szymanski showed that 
the localized endomorphism $\lambda_{u}$ is 
an inner automorphism if and only if $u$ can be written 
in some special form. As a consequence, if the 
localized endomorphism $\lambda_{u}$ is 
an inner automorphism, then 
there exists a unitary $U\in \mathcal{F}_{n}$ 
such that $\lambda_{u}={\rm Ad}U$. Keeping this in mind, 
we would like to consider the following problem. Let 
$\lambda_{u}$ and $\lambda_{v}$ be two localized endomorphisms. 
If they satisfy ${\rm Ad}U\circ\lambda_{u}=\lambda_{v}$, 
what can we say about $U$? Can we determine such a unitary $U$? 
Unfortunately in this paper 
we cannot say nothing about this problem. But we 
remark that a localized endomorphism has finite index~\cite{CP,L}. 
Therefore 
the $C^{*}$-algebras 
$\lambda_{u}(\mathcal{F}_{n})'\cap\mathcal{F}_{n}$ and 
$\lambda_{u}(\mathcal{O}_{n})'\cap\mathcal{O}_{n}$ are 
finite-dimensional. So we expect that our investigation 
would be helpful on this problem in the future. 

The author wishes to express his hearty gratitude to Professor 
Wojciech Szymanski 
for valuable comments and discussion on this 
paper. 
The author is also grateful to Professor 
Roberto Conti 
for valuable comments. 
The author would like to thank Professor 
Takeshi Katsura for useful advice 
and comments. 

\section{Main Results}
The Cuntz algebra $\mathcal{O}_{n}$ is the $C^{*}$-algebra 
generated by isometries 
$S_{1},\dots,S_{n}$ 
satisfying $\sum_{i=1}^{n}S_{i}{S_{i}}^{*}=1$. 
The gauge action $\gamma_{z}$ ($z\in {\Bbb T}$) on 
$\mathcal{O}_{n}$ is defined by $\gamma_{z}(S_{i})=zS_{i}$. 
Let 
$\mathcal{F}_{n}$ be the fixed point algebra of the gauge action. 
This algebra is isomorphic to the UHF-algebra of type 
$n^{\infty}$. So $\mathcal{F}_{n}$ has the unique tracial state 
$\tau$. 
We have a conditional expectation 
$E:\mathcal{O}_{n}\rightarrow \mathcal{F}_{n}$ defined by 
$$
E(x)=\int_{\Bbb T}\gamma_{z}(x)dz.
$$
The canonical shift $\varphi$ is defined by 
$\varphi(x)=\sum_{i=1}^{n}S_{i}xS_{i}^{*}$. 
It is easy to see that 
$S_{i}x=\varphi(x)S_{i}$ and $x{S_{i}}^{*}={S_{i}}^{*}\varphi(x)$. 
For each $x\in \mathcal{O}_{n}$, we have the Fourier expansion 
$$
x=\sum_{k=1}^{\infty}{S_{1}^{*}}^{k}x_{-k}+x_{0}+
\sum_{k=1}^{\infty}x_{k}S_{1}^{k}
$$ 
where $x_{k}=E(x{S_{1}^{*}}^{k}),\ x_{-k}=E(S_{1}^{k}x)$ 
and $x_{0}=E(x)$. 
(The right-hand side converges in the Hilbert space 
generated by the GNS-representation 
with respect to $\tau\circ E$.) 
For example, if $x=S_{1}S_{2}S_{5}^{*}{S_{8}^{*}}^{2}S_{3}^{*}$, then 
$
S_{1}^{2}x\in \mathcal{F}_{n}
$ 
and 
$
x
={S_{1}^{*}}^{2}(S_{1}^{2}x)
={S_{1}^{*}}^{2}E(S_{1}^{2}x).
$

For the inclusion of $C^{*}$-algebras $A\subset B$, the normalizer group 
is defined by 
$$
\mathcal{N}_{B}(A)=
\{u\in B\ |\ uAu^{*}=A,\ uu^{*}=u^{*}u=1\}.
$$
For a unitary operator $u$, we define the inner automorphism 
by ${\rm Ad}u(x)=uxu^{*}$. We denote by 
${\rm Ad}u|_{A}$ the restriction 
of ${\rm Ad}u$ on $A$.

\bigskip

The following two theorems are the main results 
of this paper. 

\begin{thm}
Let $A$ be a $C^{*}$-subalgebra of $\mathcal{F}_{n}$. 
If the relative commutant $A'\cap \mathcal{F}_{n}$ is 
finite-dimensional, then the algebra $A'\cap \mathcal{O}_{n}$ 
is also finite-dimensional.
\end{thm} 

\begin{thm} 
Let $A$ be as above. 
We consider two subgroups of the automorphism group 
${\rm Aut}(A)$ as follows. 
$$
G=\{{\rm Ad}U|_{A}\ \ |\ U\in \mathcal{N}_{\mathcal{O}_{n}}(A)\},\ \ 
H=\{{\rm Ad}u|_{A}\ \ |\ u\in \mathcal{N}_{\mathcal{F}_{n}}(A)\}.
$$
Then $H$ is a subgroup of $G$ with finite index. 
\end{thm}

We need some preparations to show these theorems. 

\begin{lem}
For $X\in A'\cap \mathcal{O}_{n}$, we set 
$x_{k}=E(X{S_{1}^{*}}^{k})$ and $x_{-k}=E(S_{1}^{k}X)$. 
Then for any $a\in A$ we have 
$ax_{k}=x_{k}\varphi^{k}(a)$, $x_{-k}a=\varphi^{k}(a)x_{-k}$ 
and 
$x_{k}x_{k}^{*},\ x_{-k}^{*}x_{-k}\in A'\cap \mathcal{F}_{n}$
\end{lem}

\begin{proof}
For any $a\in A$, we see that
\begin{align*}
ax_{k}&=aE(X{S_{1}^{*}}^{k})=E(aX{S_{1}^{*}}^{k})=
E(Xa{S_{1}^{*}}^{k})\\
&=E(X{S_{1}^{*}}^{k}\varphi^{k}(a))=
E(X{S_{1}^{*}}^{k})\varphi^{k}(a)
=x_{k}\varphi^{k}(a)
\end{align*}
and therefore 
$$
x_{k}x_{k}^{*}a=x_{k}(a^{*}x_{k})^{*}
=x_{k}(x_{k}\varphi^{k}(a)^{*})^{*}
=x_{k}\varphi^{k}(a)x_{k}^{*}
=ax_{k}x_{k}^{*}.
$$
In the same way we also have 
$x_{-k}a=\varphi^{k}(a)x_{-k}$ and 
$x_{-k}^{*}x_{-k}a=ax_{-k}^{*}x_{-k}$. 
\end{proof}

\begin{lem}
There is a positive integer $N$ satisfying the following 
properties. For any integer $k\geq N$ and any element 
$X\in A'\cap \mathcal{O}_{n}$, we have 
$x_{k}=E(X{S_{1}^{*}}^{k})=0$ and $x_{-k}=E(S_{1}^{k}X)=0$.

\end{lem}

\begin{proof}
We compute 
$$
x_{k}^{*}x_{k}=E(X{S_{1}^{*}}^{k})^{*}E(X{S_{1}^{*}}^{k})
\leq E(S_{1}^{k}X^{*}X{S_{1}^{*}}^{k})
\leq ||X||^{2}E(S_{1}^{k}{S_{1}^{*}}^{k})
=||X||^{2}S_{1}^{k}{S_{1}^{*}}^{k}.
$$
Let $R={\mathcal{F}_{n}}''$ be the hyperfinite ${\rm II}_{1}$-factor. 
We take the polar decomposition 
$x_{k}=v_{k}|x_{k}|$ in $R$. Then the above computation shows that 
$v_{k}^{*}v_{k}\leq S_{1}^{k}{S_{1}^{*}}^{k}$. 
On the other hand, since $x_{k}x_{k}^{*}$ is an element 
of the finite-dimensional $C^{*}$-algebra 
$A'\cap \mathcal{F}_{n}$, 
we have $v_{k}v_{k}^{*}\in A'\cap \mathcal{F}_{n}$. 
Since the $C^{*}$-algebra $A'\cap \mathcal{F}_{n}$ is finite-
dimensional, there is a positive number $c$ satisfying 
$\tau(p)\geq c$ for any non-zero projection $p\in A'\cap \mathcal{F}_{n}$. 
We can take a positive integer $N$ satisfying 
$\tau(S_{1}^{k}{S_{1}^{*}}^{k})=\frac{1}{n^{k}}< c$ for any $k\geq N$. 
Then we see that 
$
\tau(v_{k}v_{k}^{*})=\tau(v_{k}^{*}v_{k})\leq 
\tau(S_{1}^{k}{S_{1}^{*}}^{k})<c$ 
and hence $v_{k}v_{k}^{*}=0$. So we conclude that $x_{k}=0$ 
for $k\geq N$. 
In the same way we also have $x_{-k}=0$ 
for $k\geq N$.
\end{proof}

\begin{lem} 
Let $N$ be the positive integer in the previous lemma. 
For any $X\in A'\cap \mathcal{O}_{n}$, we have 
$$
X=\sum_{k=1}^{N}{S_{1}^{*}}^{k}x_{-k}+x_{0}+\sum_{k=1}^{N}x_{k}S_{1}^{k}
$$ 
where $x_{k}=E(X{S_{1}^{*}}^{k}),\ x_{-k}=E(S_{1}^{k}X)$ 
and $x_{0}=E(X)$.
\end{lem}

\begin{proof}
We have the Fourier expansion 
$$
X=\sum_{k=1}^{\infty}{S_{1}^{*}}^{k}x_{-k}+x_{0}+
\sum_{k=1}^{\infty}x_{k}S_{1}^{k}.
$$
Thus by the previous lemma, we are done. 
\end{proof}

\begin{lem}
We define the isomorphism $\pi_{k}$ on $A$ by 
$$
\pi_{k}(x)=
\begin{pmatrix}
x&0\\
0&\varphi^{k}(x)
\end{pmatrix},\ \ \ x\in A.
$$
Then we have 
$$
\begin{pmatrix}
0&x_{k}\\
x_{k}^{*}&0
\end{pmatrix},
\ 
\begin{pmatrix}
0&x_{-k}^{*}\\
x_{-k}&0
\end{pmatrix}
\in \pi_{k}(A)'\cap M_{2}(\mathcal{F}_{n})
$$
where 
$x_{k}=E(X{S_{1}^{*}}^{k}),\ x_{-k}=E(S_{1}^{k}X)$ 
for $X\in A'\cap \mathcal{O}_{n}$.
\end{lem}

\begin{proof}
This is an immediate consequence of the relations 
$ax_{k}=x_{k}\varphi^{k}(a)$ and 
$x_{-k}a=\varphi^{k}(a)x_{-k}$ for $a\in \mathcal{O}_{n}$. 
\end{proof}

\begin{lem}
The $C^{*}$-algebra $\pi_{k}(A)'\cap M_{2}(\mathcal{F}_{n})$ 
is finite-dimensional. 
\end{lem}

\begin{proof}
We set 
$$
B=\pi_{k}(A)'\cap M_{2}(\mathcal{F}_{n}),\ \ 
e=
\begin{pmatrix}
1&0\\
0&0
\end{pmatrix}
\in B,\ \ 
f=
\begin{pmatrix}
0&0\\
0&1
\end{pmatrix}
=1-e.
$$
Then we see that $eBe\simeq A'\cap \mathcal{F}_{n}$ and 
$fBf\simeq \varphi^{k}(A)'\cap \mathcal{F}_{n}
\simeq M_{n^{k}}({\Bbb C})\otimes (A'\cap \mathcal{F}_{n})
$. So the both algebra $eBe$ and $fBf$ are finite-dimensional 
and hence $B$ is finite-dimensional. Indeed 
the von Neumann algebras $eB''e$ and $fB''f$ are finite-dimensional. 
So the center of $B''$ is finite-dimensional. Therefore we 
may assume that $B''$ is a factor. Then 
the finite-dimensionality of $eB''e$ and $fB''f$ ensures that 
$B''$ is finite-dimensional. Therefore $B$ is finite-dimensional. 

\end{proof}

\bigskip

\begin{proof}[Proof of Theorem 2.1.]
Consider the vector space 
$$V_{k}=\{x_{k}=E(X{S_{1}^{*}}^{k})\ |\ X\in A'\cap \mathcal{O}_{n}\}.$$ 
Since the map 
$$
V_{k}\ni x\mapsto 
\begin{pmatrix}
0&x\\
x^{*}&0
\end{pmatrix}
\in \pi_{k}(A)'\cap M_{2}(\mathcal{F}_{n})
$$ 
is injective and $\Bbb{R}$-linear, the vector space 
$V_{k}$ is finite-dimensional. 
In the same way the vector space 
$$
V_{-k}=\{x_{k}=E(S_{1}^{k}X)\ |\ X\in A'\cap \mathcal{O}_{n}\}
$$ 
is also finite-dimensional. 
On the other hand, the element $x_{0}=E(X)$ belongs to 
the finite-dimensional $C^{*}$-algebra $A'\cap \mathcal{F}_{n}$. 
Combining these with Lemma 2.5, we see that 
$A'\cap \mathcal{O}_{n}$ is finite-dimensional. 
\end{proof}

\bigskip

\begin{prop}
There exists an orthogonal family of minimal projections 
$e_{1},\dots,e_{l}\in A'\cap \mathcal{O}_{n}$ satisfying the 
following. 
\begin{enumerate}
\item $\sum_{i=1}^{l}e_{i}=1$ and 
$e_{1},\dots,e_{l}\in A'\cap \mathcal{F}_{n}$. 
\item There are integers $k_{1},\dots,k_{l}$ such that 
${\rm Ad}u_{z}(x)=\gamma_{z}(x)$ 
for $x\in A'\cap \mathcal{O}_{n}$ 
where 
$u_{z}=z^{k_{1}}e_{1}+\cdots+z^{k_{l}}e_{l}$. 
\end{enumerate}
\end{prop}

\begin{proof}
Since $A'\cap \mathcal{O}_{n}$ is finite-dimensional and 
$\gamma$-invariant, there exists an orthogonal 
family of minimal projections 
$e_{1},\dots,e_{l}\in A'\cap \mathcal{O}_{n}$ and 
integers $k_{1},\dots,k_{l}$ such that 
${\rm Ad}u_{z}(x)=\gamma_{z}(x)$ where 
$u_{z}=z^{k_{1}}e_{1}+\cdots+z^{k_{l}}e_{l}$ and 
$x\in A'\cap \mathcal{O}_{n}$. 
Then $e_{i}\in (A'\cap \mathcal{O}_{n})^{\gamma}
=A'\cap \mathcal{F}_{n}$. 
\end{proof}

\begin{cor}
If $A$ is an irreducible $C^{*}$-subalgebras of 
$\mathcal{F}_{n}$, then we have 
$A'\cap \mathcal{O}_{n}={\Bbb C}$. 
\end{cor}

\begin{proof}
By the previous proposition, we know that 
there are minimal projections in $A'\cap \mathcal{O}_{n}$ 
such that they belong to $A'\cap \mathcal{F}_{n}$. 
Thus we are done. 
\end{proof}

In the rest of this paper we frequently use the projections 
$e_{1},\dots,e_{l}\in A'\cap \mathcal{F}_{n}$ and the unitary 
$u_{z}=z^{k_{1}}e_{1}+\cdots+z^{k_{l}}e_{l}$ 
in the above proposition.

\begin{rem}
The Bratteli diagram of 
the inclusion 
$A'\cap \mathcal{F}_{n}\subset A'\cap \mathcal{O}_{n}$ has a 
special form. They have a common family of minimal projections. So 
for each vertex corresponding to a direct summand of 
$A'\cap \mathcal{F}_{n}$, 
there is only one edge which starts on this vertex. For example, 
if $A'\cap \mathcal{F}_{n}={\Bbb C}$, then 
$A'\cap \mathcal{O}_{n}={\Bbb C}$. If 
$A'\cap \mathcal{F}_{n}={\Bbb C}\oplus {\Bbb C}$, then 
$A'\cap \mathcal{O}_{n}$ is isomorphic to either 
${\Bbb C}\oplus {\Bbb C}$ or $M_{2}({\Bbb C})$. 
\end{rem}

\begin{lem} 
Let $U\in \mathcal{O}_{n}$ be a unitary satisfying 
$UAU^{*}\subset \mathcal{F}_{n}$. Then we have 
\begin{enumerate} 
\item 
$U^{*}\gamma_{z}(U)\in A'\cap \mathcal{O}_{n}$. 
\item 
There exists a unitary $w\in A'\cap \mathcal{O}_{n}$ 
and integers $m_{1},\dots,m_{l}$ such that 
$\gamma_{z}(Uwe_{i})=z^{m_{i}}Uwe_{i}$
\end{enumerate}
\end{lem}

\begin{proof}
For any $a\in A$, since $UaU^{*}\in \mathcal{F}_{n}$, 
we see that 
$$
\gamma_{z}(U)a\gamma_{z}(U)^{*}=\gamma_{z}(UaU^{*})
=UaU^{*}.
$$ 
Thus we have $U^{*}\gamma_{z}(U)\in A'\cap \mathcal{O}_{n}$. 
It is easy to see that the family 
$\{U^{*}\gamma_{z}(U)u_{z}\}_{z\in {\Bbb T}}$ is a 
unitary group. Indeed since $U^{*}\gamma_{z}(U)\in A'\cap \mathcal{O}_{n}$, 
we see that 
\begin{align*}
U^{*}\gamma_{z_{1}}(U)u_{z_{1}}&U^{*}\gamma_{z_{2}}(U)u_{z_{2}}
=U^{*}\gamma_{z_{1}}(U){\rm Ad}u_{z_{1}}(U^{*}\gamma_{z_{2}}(U))
u_{z_{1}z_{2}}\\
&=U^{*}\gamma_{z_{1}}(U)\gamma_{z_{1}}(U^{*})\gamma_{z_{1}z_{2}}(U)
u_{z_{1}z_{2}}
=U^{*}\gamma_{z_{1}z_{2}}(U)u_{z_{1}z_{2}}. 
\end{align*}
Since $\{U^{*}\gamma_{z}(U)u_{z}\}_{z\in {\Bbb T}}$ is a 
unitary group in the finite-dimensional $C^{*}$-algebra 
$A'\cap \mathcal{O}_{n}$, we can take a unitary 
$w\in A'\cap \mathcal{O}_{n}$ 
and integers $n_{1},\dots,n_{l}$ such that 
$w^{*}U^{*}\gamma_{z}(U)u_{z}w=z^{n_{1}}e_{1}+\cdots+z^{n_{l}}e_{l}$. 
Then we see that 
\begin{align*}
\gamma_{z}(Uwe_{i})&=\gamma_{z}(U)u_{z}wu_{z}^{*}e_{i}
=\{Uw(z^{n_{1}}e_{1}+\cdots+z^{n_{l}}e_{l})w^{*}u_{z}^{*}\}
u_{z}wu_{z}^{*}e_{i}\\
&=Uw(z^{n_{1}}e_{1}+\cdots+z^{n_{l}}e_{l})u_{z}^{*}e_{i}
=z^{n_{i}-k_{i}}Uwe_{i}.
\end{align*}
\end{proof}

\begin{rem}
By the previous lemma, we know that the Fourier expansion 
of $U$ can be write down in a finite sum. Indeed if $m_{i}>0$, then 
$Uwe_{i}=(Uwe_{i}{S_{1}^{*}}^{m_{i}})S_{1}^{m_{i}}$ and 
$Uwe_{i}{S_{1}^{*}}^{m_{i}}\in \mathcal{F}_{n}$. 
On the other hand if $m_{i}<0$, then 
$Uwe_{j}={S_{1}^{*}}^{-m_{i}}(S_{1}^{-m_{i}}Uwe_{j})$ and 
$S_{1}^{-m_{i}}Uwe_{j}\in \mathcal{F}_{n}$. Therefore 
the Fourier expansion 
of $Uw$ is a finite sum. Combining this with Lemma 2.5, 
we can show that the Fourier expansion 
of $U$ is a finite sum.
\end{rem}

\begin{prop}
For any normalizer 
$U\in \mathcal{N}_{\mathcal{O}_{n}}(A)$, 
there exist unitaries $v\in A'\cap \mathcal{F}_{n}$ 
and $w\in A'\cap \mathcal{O}_{n}$ satisfying 
$$
vUw\in \mathcal{N}_{\mathcal{O}_{n}}
(e_{1}\mathcal{F}_{n}e_{1}
\oplus\cdots\oplus e_{l}\mathcal{F}_{n}e_{l}).
$$
\end{prop}

\begin{proof}
By the previous lemma we have 
$\gamma_{z}(Uwe_{i})=z^{m_{i}}Uwe_{i}$. Then we get 
$
\gamma_{z}(Uwe_{i}w^{*}U^{*})=Uwe_{i}w^{*}U^{*}
$ and hence 
$Uwe_{i}w^{*}U^{*}\in \mathcal{F}_{n}$. 
Since 
$UA'\cap \mathcal{O}_{n}U^{*}=A'\cap \mathcal{O}_{n}$, 
we have 
$Uwe_{i}w^{*}U^{*}\in A'\cap\mathcal{F}_{n}$. 
Thus $\{Uwe_{i}w^{*}U^{*}\}_{i}$ is a family of 
minimal projections in the finite-dimensional 
$C^{*}$-algebra $A'\cap\mathcal{F}_{n}$. So 
we can find a unitary $v\in A'\cap \mathcal{F}_{n}$ 
satisfying $vUwe_{i}w^{*}U^{*}v^{*}=e_{j}$. 
Since $\gamma_{z}(vUwe_{i})=v\gamma_{z}(Uwe_{i})
=z^{m_{i}}vUwe_{i}$ for any $x\in \mathcal{F}_{n}$, 
we see that 
$
\gamma_{z}\circ{\rm Ad}vUw(e_{i}xe_{i})=
{\rm Ad}vUw(e_{i}xe_{i})
$. Therefore 
$(vUw)e_{i}\mathcal{F}_{n}e_{i}(vUw)^{*}
\subset e_{j}\mathcal{F}_{n}e_{j}
$. On the other hand, 
since 
$
\gamma_{z}(w^{*}U^{*}v^{*}e_{j})=
\gamma_{z}(vUwe_{i})^{*}=(z^{m_{i}}Uwe_{i})^{*}
=z^{-m_{i}}w^{*}U^{*}v^{*}e_{j}
$, 
we also have $(vUw)^{*}e_{j}\mathcal{F}_{n}e_{j}(vUw)
\subset e_{i}\mathcal{F}_{n}e_{i}
$. Therefore we have 
$$
{\rm Ad}vUw(e_{1}\mathcal{F}_{n}e_{1}
\oplus\cdots\oplus e_{l}\mathcal{F}_{n}e_{l})
\subset e_{1}\mathcal{F}_{n}e_{1}
\oplus\cdots\oplus e_{l}\mathcal{F}_{n}e_{l},
$$

$$
{\rm Ad}w^{*}U^{*}v^{*}(e_{1}\mathcal{F}_{n}e_{1}
\oplus\cdots\oplus e_{l}\mathcal{F}_{n}e_{l})
\subset e_{1}\mathcal{F}_{n}e_{1}
\oplus\cdots\oplus e_{l}\mathcal{F}_{n}e_{l}
$$
and hence 
$$
vUw\in \mathcal{N}_{\mathcal{O}_{n}}
(e_{1}\mathcal{F}_{n}e_{1}
\oplus\cdots\oplus e_{k}\mathcal{F}_{n}e_{k}).
$$
\end{proof}

\begin{rem}
The normalizer $\mathcal{N}_{\mathcal{O}_{n}}
(\mathcal{F}_{n})$ is a subset of $\mathcal{F}_{n}$. 
However the structure of $\mathcal{N}_{\mathcal{O}_{n}}
(e_{1}\mathcal{F}_{n}e_{1}
\oplus\cdots\oplus e_{l}\mathcal{F}_{n}e_{l})$ is not 
simple in general. See Examples 2.1 and 2.2. 
\end{rem}

\begin{lem}
Let $e\in \mathcal{F}_{n}$ be a projection. If a partial isometry 
$u\in \mathcal{O}_{n}$ satisfies $u^{*}u=uu^{*}=e$ and 
$ue\mathcal{F}_{n}eu^{*}=e\mathcal{F}_{n}e$, then we have 
$u\in \mathcal{F}_{n}$. 

\end{lem}

\begin{proof}
Since $u^{*}\gamma_{z}(u)\in (e\mathcal{F}_{n}e)'\cap e\mathcal{O}_{n}e
=e(\mathcal{F}_{n}'\cap \mathcal{O}_{n})e={\Bbb C}e
$, we have $\gamma_{z}(u)=z^{m}u$ for some integer $m$. 
We will show $m=0$. Suppose that $m>0$. 
Set $v=u{S_{1}^{*}}^{m}$. 
Then we have 
$\gamma_{z}(v)=v$ and 
hence $v\in \mathcal{F}_{n}$. Then we compute 
$v^{*}v=S_{1}^{m}e{S_{1}^{*}}^{m}=\varphi^{m}(e)S_{1}^{m}{S_{1}^{*}}^{m}$ 
and 
$
vv^{*}=uu^{*}=e
$. So we see that 
$
\tau(e)=\tau(vv^{*})=\tau(v^{*}v)=\tau(\varphi^{m}(e)S_{1}^{m}{S_{1}^{*}}^{m})
=\tau(e)\times \tau(S_{1}^{m}{S_{1}^{*}}^{m})=\frac{1}{n^{m}}\tau(e)
<\tau(e).
$ This is a contradiction. On the other hand, if $m<0$, we 
have $\gamma_{z}(u^{*})=z^{-m}u^{*}$. So by the same way we get a 
contradiction.

\end{proof}

\begin{lem}
Let $B$ be the abelian $C^{*}$-algebra 
generated by $e_{1},\dots,e_{l}$. 
Then we have 
$$
\mathcal{N}_{\mathcal{O}_{n}}(A)
\cap \mathcal{N}_{\mathcal{O}_{n}}(B)
\subset \mathcal{N}_{\mathcal{O}_{n}}
(e_{1}\mathcal{F}_{n}e_{1}
\oplus\cdots\oplus e_{k}\mathcal{F}_{n}e_{k})
.$$ 
\end{lem}

\begin{proof}
The proof is essentially same as that of 
Lemma 2.10 and Proposition 2.11. 
For any $U\in \mathcal{N}_{\mathcal{O}_{n}}(A)
\cap \mathcal{N}_{\mathcal{O}_{n}}(B)$, 
since $U\in \mathcal{N}_{\mathcal{O}_{n}}(B)$, 
we have $U^{*}\gamma_{z}(U)\in B'$ and hence 
$U^{*}\gamma_{z}(U)u_{z}\in B'$. Therefore 
we can take $w=1$ in the proof of Lemma 2.10. 
Then since $U\in \mathcal{N}_{\mathcal{O}_{n}}(B)$, 
we have $Uwe_{i}w^{*}U^{*}=Ue_{i}U^{*}=e_{j}$ 
and hence we can take $v=1$ in the proof of Proposition 2.11. 
Thus by Proposition 2.11, we have 
$U\in \mathcal{N}_{\mathcal{O}_{n}}
(e_{1}\mathcal{F}_{n}e_{1}
\oplus\cdots\oplus e_{k}\mathcal{F}_{n}e_{k})$. 
\end{proof}

\bigskip

\begin{proof}[Proof of Theorem 2.2.] 
We can choose a finite family of unitaries 
$U_{1},\dots,U_{N}\in \mathcal{N}_{\mathcal{O}_{n}}(A)
\cap \mathcal{N}_{\mathcal{O}_{n}}(B)\subset \mathcal{N}_{\mathcal{O}_{n}}
(e_{1}\mathcal{F}_{n}e_{1}
\oplus\cdots\oplus e_{k}\mathcal{F}_{n}e_{k})$ satisfying the following. 
For any $V\in \mathcal{N}_{\mathcal{O}_{n}}(A)
\cap \mathcal{N}_{\mathcal{O}_{n}}(B)$, there exists $U_{i}$ such that 
${\rm Ad}V={\rm Ad}U_{i}$ on $B$. 

For any $U\in \mathcal{N}_{\mathcal{O}_{n}}(A)$, by Proposition 2.11 
there exist unitaries $v\in A'\cap \mathcal{F}_{n}$ 
and $w\in A'\cap \mathcal{O}_{n}$ satisfying 
$vUw\in \mathcal{N}_{\mathcal{O}_{n}}
(e_{1}\mathcal{F}_{n}e_{1}
\oplus\cdots\oplus e_{k}\mathcal{F}_{n}e_{k})$. Then since 
$vUw\in \mathcal{N}_{\mathcal{O}_{n}}(A)
\cap \mathcal{N}_{\mathcal{O}_{n}}(B)$, 
we can take $U_{i}$ satisfying 
${\rm Ad}U_{i}^{*}vUw={\rm id}$ on $B$. Combining this with 
the fact that $U_{i}\in \mathcal{N}_{\mathcal{O}_{n}}
(e_{1}\mathcal{F}_{n}e_{1}
\oplus\cdots\oplus e_{k}\mathcal{F}_{n}e_{k})$ we see that 
$U_{i}^{*}vUwe_{j}\in \mathcal{N}_{e_{j}\mathcal{O}_{n}e_{j}}
(e_{j}\mathcal{F}_{n}e_{j})\subset \mathcal{F}_{n}$ 
and hence $U_{i}^{*}vUw\in \mathcal{N}_{\mathcal{F}_{n}}(A)$. 
Here we used Lemma 2.12. 
Therefore we see that ${\rm Ad}U|_{A}={\rm Ad}vUw|_{A}
\in ({\rm Ad}U_{i}|_{A})H
$. This implies that the index $[G:H]$ is finite. 

\end{proof} 

\bigskip

\begin{ex}
Let $e$ be a projection in 
$\mathcal{F}_{n}$. 
Consider the $C^{*}$-algebra 
$
A=e\mathcal{F}_{n}e\oplus(1-e)\mathcal{F}_{n}(1-e)
$. Here we remark that $A'\cap \mathcal{F}_{n}={\Bbb C}e
\oplus{\Bbb C}(1-e).$ 
We will show that 
$
\mathcal{N}_{\mathcal{O}_{n}}(A)\subset \mathcal{F}_{n}
$ and hence $G=H$. This can be shown by K-theoretic 
argument as follows. 

For any $U\in \mathcal{N}_{\mathcal{O}_{n}}(A)$, 
if $UeU^{*}=e$, it follows from Lemma 2.12 that 
$U\in \mathcal{F}_{n}$. So we consider the case 
$UeU^{*}=1-e$. 
Since $U^{*}\gamma_{z}(U)e\in (e\mathcal{F}_{n}e)'\cap e{O}_{n}e
=e(\mathcal{F}_{n}'\cap {O}_{n})e={\Bbb C}e
$, we have $\gamma_{z}(U)e=z^{m}Ue$ for some integer $m$. 
We will show $m=0$. Suppose that $m>0$. 
Set $v=Ue{S_{1}^{*}}^{m}$. 
Then we have 
$\gamma_{z}(v)=v$ and 
hence $v\in \mathcal{F}_{n}$. Then we compute 
$v^{*}v=S_{1}^{m}e{S_{1}^{*}}^{m}=\varphi^{m}(e)S_{1}^{m}{S_{1}^{*}}^{m}$ 
and 
$
vv^{*}=UeU^{*}=1-e
$. So we see that 
$
1-\tau(e)=
\tau(vv^{*})=\tau(v^{*}v)=\tau(\varphi^{m}(e)S_{1}^{m}{S_{1}^{*}}^{m})
=\tau(e)\times \tau(S_{1}^{m}{S_{1}^{*}}^{m})=\frac{1}{n^{m}}\tau(e)
$. Since $\mathcal{F}_{n}$ is the UHF-algebra of type $n^{\infty}$, 
we can write $\tau(e)=\frac{q}{n^{p}}$. So we get 
$
1-\frac{q}{n^{p}}=\frac{1}{n^{m}}\frac{q}{n^{p}}
$ and hence 
$$
n^{m+p}=q(1+n^{m}).
$$
This is impossible. Indeed, consider the prime factorization 
$n=p_{1}^{k_{1}}\times \cdots\times p_{n}^{k_{n}}$. Then we have 
$$
(p_{1}^{k_{1}}\times \cdots\times p_{n}^{k_{n}})^{m+p}=
q(1+(p_{1}^{k_{1}}\times \cdots\times p_{n}^{k_{n}})^{m}).
$$ 
Therefore we must have 
$$
1+(p_{1}^{k_{1}}\times \cdots\times p_{n}^{k_{n}})^{m}
=p_{1}^{l_{1}}\times \cdots\times p_{n}^{l_{n}}.
$$
However this cannot occur because the left hand side has 
the remainder $1$ when dividing by $p_{1}$. 
\end{ex}

\begin{ex}
We can write 
$$
\mathcal{O}_{2}\supset\mathcal{F}_{2}=
M_{2}({\Bbb C})\otimes M_{2}({\Bbb C})\otimes\cdots.
$$
Consider two projections 
$$
e=
\begin{pmatrix}
1&0\\
0&0
\end{pmatrix}
\otimes 1\otimes 1\otimes\cdots
$$
and 
$$
f=
\begin{pmatrix}
0&0\\
0&1
\end{pmatrix}
\otimes 
\begin{pmatrix}
1&0\\
0&0
\end{pmatrix}
\otimes 1\otimes\cdots.
$$
Since $\varphi(e)S_{1}{S_{1}}^{*}\in M_{2}({\Bbb C})\otimes M_{2}({\Bbb C})$ 
and 
$\tau(\varphi(e)S_{1}{S_{1}}^{*})=\frac{1}{4}$, 
there exists a partial isometry 
$v\in M_{2}({\Bbb C})\otimes M_{2}({\Bbb C})$ such that 
$v^{*}v=\varphi(e)S_{1}{S_{1}}^{*}$ and 
$vv^{*}=f$. We set 
$$
U=vS_{1}+(vS_{1})^{*}+(1-e-f).
$$
Then it is easy to see that 
$$
U\in \mathcal{N}_{\mathcal{O}_{n}}
(e\mathcal{F}_{n}e
\oplus f\mathcal{F}_{n}f
\oplus (1-e-f)\mathcal{F}_{n}(1-e-f)
).
$$
We let 
$
A=e\mathcal{F}_{n}e
\oplus f\mathcal{F}_{n}f
\oplus (1-e-f)\mathcal{F}_{n}(1-e-f)
$. 
Since 
$\tau(UeU^{*})=\tau(f)=\frac{1}{4}\not=\frac{1}{2}=\tau(e)$, 
we have 
$$
{\rm Ad}U|_{A}\not\in 
\{{\rm Ad}u|_{A}\ \ |\ u\in \mathcal{N}_{\mathcal{F}_{n}}(A)\}.
$$
Therefore we see that $G\not=H$. 
\end{ex}

\begin{rem}
If $A$ is of the form 
$A=e_{1}\mathcal{F}_{n}e_{1}
\oplus\cdots\oplus e_{l}\mathcal{F}_{n}e_{l}$, 
then we have 
$
A'\cap \mathcal{F}_{n}=A'\cap \mathcal{O}_{n}
={\Bbb C}e_{1}\oplus\cdots\oplus {\Bbb C}e_{l}
$. On the other hand, in Remark 2.1 we see that 
the Bratteli diagram of 
the inclusion 
$A'\cap \mathcal{F}_{n}\subset A'\cap \mathcal{O}_{n}$ has a 
special form. So we might expect that 
$A'\cap \mathcal{F}_{n}= A'\cap \mathcal{O}_{n}$. However 
this is wrong in general. Indeed there exists a $C^{*}$-subalgebra 
$A\subset \mathcal{F}_{n}$ with finite index such that 
$A'\cap \mathcal{F}_{n}\not= A'\cap \mathcal{O}_{n}$. 
We can take $A=\lambda_{u}(\mathcal{F}_{n})$ where 
$\lambda_{u}$ is a localized endomorphism. 
See~\cite{CHS,CP}.

\end{rem}

\end{document}